\documentclass[12pt]{amsart}

\usepackage{amsthm, amssymb, amstext, amscd, amsfonts, amsxtra, latexsym, amsmath, comment, 
 mathrsfs, esint, stmaryrd}
\usepackage{color}
\usepackage[svgnames]{xcolor}
\usepackage{fancybox}
\usepackage{fullpage}
\usepackage[english]{babel}
\usepackage[latin1]{inputenc}

\numberwithin{equation}{section}
\newtheorem{theorem}{Theorem}
\numberwithin{theorem}{section}
\newtheorem{lemma}[theorem]{Lemma}
\newtheorem{proposition}[theorem]{Proposition}
\newtheorem{corollary}[theorem]{Corollary}

\newcommand{\s}{\operatorname{spt}_\omega}
\newcommand{\N}{\mathcal{N}}
\newcommand{\M}{\mathcal{M}}

\newcommand{\Z}{\mathbb{Z}}
\newcommand{\e}{\equiv}
\renewcommand{\(}{\left(}
\renewcommand{\)}{\right)}

\renewenvironment{proof}[1][Proof]{\begin{trivlist} \item[\hskip \labelsep {\bfseries #1:}]}{\qed\end{trivlist}}

\author{Min-Joo Jang}
\address{Mathematical Institute\\University of Cologne\\ Weyertal 86-90 \\ 50931 Cologne \\Germany}
\email{min-joo.jang@uni-koeln.de}

\author{Byungchan Kim}
\address{School of Liberal Arts \\ Seoul National University of Science and Technology \\ 232 Gongneung-ro, Nowon-gu, Seoul 01811, Korea}
\email{bkim4@seoultech.ac.kr}

\date{\today}
\subjclass[2010] {11P82}
\keywords{spt function, spt-crank function, asymptotic formula, partial theta function, congruences}
\thanks{Byungchan Kim was supported by the Basic Science Research Program through the National Research Foundation of Korea (NRF) funded by the Ministry of Education (NRF-2013R1A1A2061326).}

\begin{document}

\title{On spt-crank type functions}

\begin{abstract}
In a recent paper, Andrews, Dixit, and Yee introduced a new spt-type function $\s(n)$, which is closely related to Ramanujan's third order mock theta function $\omega(q)$. Garvan and Jennings-Shaffer introduced a crank function which explains congruences for $\s(n)$.  In this note, we study asymptotic behavior of this crank function and confirm a positivity conjecture of the crank asymptotically. We also study a sign pattern of the crank and congruences for $\s (n)$. 
\end{abstract}

\maketitle

\section{Introduction}

 Since Andrews \cite{GA} introduced the spt function, which counts the total number of appearance of the smallest part in each integer partition of $n$, there have been numerous studies on spt function and its variants. For example, see \cite{ack, AGL, fo, gar1, gar2, rol} to name a few.  In particular, Andrews proves striking congruences:
\begin{align*}
\operatorname{spt}( 5n+4) &\equiv 0 \pmod{5}, \\
\operatorname{spt}( 7n+5 ) &\equiv 0 \pmod{7}, \\
\operatorname{spt}( 13n + 6 ) & \equiv 0 \pmod{13}.
\end{align*}
Motivated from Dyson's rank \cite{FD} and Andrews and Garvan's crank \cite{AG} which explain Ramanujan's famous partition congruences,  Andrews, Garvan, and Liang \cite{AGL} introduced an spt-crank which explains the modulo 5 and modulo 7 congruences of the spt function.

 More recently,  Andrews, Dixit, and Yee \cite{ADY} introduced a new spt function $\s (n)$. The partition function $p_{\omega} (n)$ is defined to be the number of partitions of $n$ such that all odd parts are smaller than twice the smallest part. A new spt-type function $\s(n)$ is defined by the total number of appearances of the smallest part in each partition enumerated by $p_\omega(n)$.  Recall that Ramanujan's third order mock theta function is
$$
\omega(q):=\sum_{n=0}^{\infty}\frac{q^{2n^2+2n}}{(q;q^2)_{n+1}^2}.
$$
The subscript $\omega$ is used in $p_\omega$ because its generating function is essentially $\omega(q)$ \cite[Theorem 3.1]{ADY} :
\[
\sum_{n\ge1} p_\omega(n)q^n=\sum_{n\ge1}\frac{q^n}{\left(1-q^n\right)\left(q^{n+1};q\right)_n\left(q^{2n+2};q^2\right)_\infty}=q\omega(q).
\]
As usual, $(a)_n:=(a;q)_n := \prod_{k=1}^{n} (1-aq^{k-1})$ for $n \in \mathbb{N}_0 \cup \{ \infty \}$. In particular, Andrews, Dixit, and Yee proved the congruence
\begin{align*}
\s (5n+3) \equiv 0 \pmod{5}.
\end{align*}

 On the other hand, motivated by Andrews, Garvan and Liang \cite{AGL}, Garvan and Jennings-Shaffer \cite{GJ}  introduced many spt like functions with corresponding spt-crank-type functions. Garvan and Jennings-Shaffer first find  a proper spt-crank-type function which dissects appropriately, and later define corresponding spt-type functions which have congruences  inherited from beautiful dissections of crank functions. To introduce new spt-crank-type functions, Garvan and Jennings-Shaffer investigate Bailey pairs in Slater's list. In particular, they defined $N_{C_1} (m,n)$ by
\[
\sum_{\substack{n\geq 0 \\ m \in \mathbb{Z} }} N_{C_1} (m,n) z^m q^n =\frac{(q;q^2)_{\infty} (q)_{\infty}}{(z)_\infty(z^{-1})_{\infty}} \sum_{n=1}^{\infty} \frac{q^n (z)_n(z^{-1})_{n}}{(q;q^2)_{n} (q)_{n}}, 
\]
and showed that
\[
N_{C_1} (0,5,5n+3) = N_{C_1} (1,5,5n+3) = \cdots = N_{C_1} (4,5,5n+3),
\]
where 
\[
N_{C_1} (i, 5, n ) = \sum_{\substack{m \in \mathbb{Z} \\ m \equiv i \pmod{5}}} N_{C_1} (m,n).
\]
This clearly implies that 
\[
\operatorname{spt}_{C_1} (5n+3) \equiv 0 \pmod{5},
\]
where $\operatorname{spt}_{C_1} (n) = \sum_{m\in \mathbb{Z}} N_{C_1} (m,n)$. Actually, the generating function for $\operatorname{spt}_{C_1} (n)$ is identical with that of $\s (n)$, and thus we see that $\operatorname{spt}_{C_1} (n) = \s  (n)$ and $N_{C_1} (m,n)$ can be regarded as a crank function for $\s (n)$.

 In this paper, we investigate arithmetic properties of $\s (n)$ and its crank function $N_{C_1} (m,n)$.  The main result is an asymptotic formula for $N_{C_1} (m,n)$ which we derive by using Wright's circle method. Typically,  Wright's circle method is employed only when the generating function has just one dominant pole. In our case, due to the presence of the factor $\frac{1}{(q^2;q^2)_{\infty}}$ in the generating function (see Section 2 for details), there are two dominant poles, namely $q=\pm 1$.

\begin{theorem}\label{main1thm}
As $n \to \infty$,
\[
N_{C_1} (m,n)  \sim \frac{\log2}{4\pi \sqrt{n} } e^{\frac{\pi\sqrt{n}}{\sqrt{3}}}.
\]
\end{theorem}

The following corollary is an immediate result from Theorem \ref{main1thm} and confirms Garvan and Jennings-Shaffer's positivity conjecture on $N_{C_1} (m,n)$ asymptotically.

\begin{corollary}
For a fixed integer $m$, 
\[
N_{C_1} (m,n) > 0,
\]
for large enough integers $n$.
\end{corollary}

Bringmann and the second author \cite{BK} proved that various unimodal ranks satisfy inequalities of the form $u(m,n) > u(m+1,n)$ for large enough integers $n$. For the spt-crank $N_{C_1} (m,n)$, the situation is slightly different. 
\begin{theorem} \label{main2thm}
For a fixed nonnegative integer $m$,
\[
(-1)^{m+n+1} (N_{C_1} (m,n) - N_{C_1} (m+1, n) ) > 0,
\]
for large enough integers $n$.
\end{theorem}

We also investigate a congruence property  of $\operatorname{spt}_{\omega} (n)$ via the mock modularity of its generating function.

\begin{theorem}\label{main3thm}
Suppose that $p\ge5$ is an odd prime, and $j,m$ and $n$ are positive integers with $\big(\frac{n}{p}\big)=-1$. If $m$ is sufficiently large, then there are infinitely many primes $Q\e-1\pmod{576p^j}$ satisfying
\[
\s\left(\frac{Q^3p^mn+1}{12}\right)\e0\pmod{p^j}.
\]
\end{theorem}

  The rest of paper is organized as follows. In Section 2, we derive the generating functions for $N_{C_1} (m,n)$ and its companion $N_{C_5} (m,n)$. In Section 3, we define an auxiliary function and investigate its asymptotic behavior near and away from dominant poles. These estimate will play important roles in Section 4, where we employ Wright's circle method to prove Theorem \ref{main1thm}. In Section 5, we prove Theorem \ref{main2thm}. We conclude the paper with the proof of Theorem \ref{main3thm} in Section 6.

\section*{Acknowledgements}
This paper will be a part of the first author's PhD thesis. The authors thank Kathrin Bringmann, Michael Woodbury, and the referee for valuable comments on an earlier version of this paper.

\section{Generating functions and combinatorics}
From \cite[Prop 5.1]{GJ}, we know that
\[
S_{C_1}(z,q) := \sum_{\substack{n \geq 0 \\ m \in  \mathbb{Z}}} N_{C_1} (m,n) z^m q^n = \frac{1}{(1-z)(1-z^{-1})} ( R(z,q^2) - (q;q^2)_{\infty} C(z,q) ),
\]
where $R(z,q)$ and $C(z,q)$ are the generating functions for ordinary partition ranks and cranks. From the Lambert series expansion of $C(z,q)$ and $R(z,q)$, we obtain that
\begin{align*}
S_{C_1} (z,q) =\frac{1}{(q^2 ;q^2)_{\infty}} \sum_{n=1}^{\infty} (-1)^{n-1} \left( \frac{q^{n(n+1)/2 } (1+q^{n})}{(1-zq^{n})(1-q^{n}/z)} -  \frac{q^{3n^2 + n} (1+q^{2n})}{(1-zq^{2n})(1-q^{2n}/z)}  \right).
\end{align*}
By noting that
\[
\frac{ 1+q^{n} }{(1-zq^{n})(1-q^{n}/z)} = \frac{1}{1-q^n} \left( \frac{1}{1-zq^n} + \frac{q^n /z}{1-q^n /z} \right),
\]
we deduce that
\begin{equation} \label{genSC1}
\begin{aligned} 
S_{C_1, m} (q) :&= \sum_{n \geq 0} N_{C_1} (m,n) q^n \\
&= \frac{1}{(q^2 ; q^2)_{\infty}} \sum_{n \geq 1} (-1)^{n-1} \( \frac{q^{n(n+1)/2 + |m|n}}{1-q^n} - \frac{q^{3n^2 + n +2|m|n}}{1-q^{2n}} \).
\end{aligned}
\end{equation}

 Garvan and Jennings-Shaffer \cite{GJ} also conjectured the positivity of  $N_{C_5} (m,n)$, which is a crank to  explain the congruence $\operatorname{spt}_{C_5} (5n+3) \equiv 0 \pmod{5}$. They showed that  $\operatorname{spt}_{C_5} (n) = \sum_{m \in \mathbb{Z}} N_{C_5} (m,n)  = \s (n) - \operatorname{spt}(n/2)$, where $\operatorname{spt}(n/2) =0$ for odd $n$.  As the shape of the generating function is similar to $N_{C_1} (m,n)$, we also investigate this function. From \cite{GJ}, and after some manipulations, we find that 
\begin{equation} \label{genSC2}
\begin{aligned}
S_{C_5,m} :&= \sum_{n \geq 0} N_{C_5} (m,n) q^n \\
&= \frac{1}{(q^2 ; q^2)_{\infty}} \sum_{n \geq 1} (-1)^{n-1} \( \frac{q^{n(n+1)/2 + |m|n}}{1-q^n} - \frac{q^{n^2 + n +2|m|n}}{1-q^{2n}} \).
\end{aligned}
\end{equation}

\section{An auxiliary function}

The generating functions in the previous section suggest defining
\[
h_{A,B} (q): = \sum_{n \geq 1} (-1)^n \frac{ q^{An^2 + Bn}}{1-q^n},
\]
where $2A \in \mathbb{N}$ and $2B \in \mathbb{Z}$ with $A+B$ is a positive integer. The key step to prove Theorem \ref{main1thm} is estimating its asymptotic behaviors near dominant poles, namely $q=1$ and $q=-1$, and away from them. 

We start with investigating $h_{A,B} (q)$ near $q=1$. Throughout the paper, we set $q=e^{2\pi iz}$ with $z=x+iy$. From the Mittag-Leffler partial fraction decomposition (\cite[eqn. (3.1)]{KB2} with corrected signs), for $w\in\mathbb{C}$ we find that
\begin{equation}\label{Mittag}
\frac{e^{\pi iw}}{1-e^{2\pi iw}}=\frac{1}{-2\pi iw}+\frac{1}{-2\pi i}\sum_{k\ge1}(-1)^k\(\frac{1}{w-k}+\frac{1}{w+k}\).
\end{equation}
Using this, we rewrite $h_{A,B}(q)$ as 
\begin{align*}
h_{A,B} (q) &= \sum_{n \geq 1} (-1)^n \frac{ q^{An^2 + Bn}}{1-q^n}\\
&=\sum_{n\ge1}(-1)^n q^{An^2+\left(B-\frac12\right)n}\frac{q^{\frac{n}{2}}}{1-q^n}\\
&=\sum_{n\ge1}(-1)^n q^{An^2+\left(B-\frac12\right)n}\left(\frac{1}{-2\pi inz}+\frac{1}{-2\pi i}\sum_{k\ge1}(-1)^k\left(\frac{1}{nz-k}+\frac{1}{nz+k}\right)\right)\\
&=\frac{-1}{2\pi iz}\sum_{n\ge1}(-1)^n n^{-1}q^{An^2+\left(B-\frac12\right)n}\\
&\qquad\qquad -\frac{1}{2\pi i}\sum_{n\ge1}(-1)^n q^{An^2+\left(B-\frac12\right)n}\sum_{k\ge1}(-1)^k\left(\frac{1}{nz-k}+\frac{1}{nz+k}\right).
\end{align*}
We first note that
$$
\frac{1}{nz-k}+\frac{1}{nz+k}=\frac{2nz}{n^2z^2-k^2}=2nz\(\frac{1}{n^2z^2-k^2}+\frac{1}{k^2}-\frac{1}{k^2}\),
$$
and for $|x|\le y$
$$
\sum_{k\ge1}\left|\frac{1}{n^2z^2-k^2}+\frac{1}{k^2}\right|=\sum_{k\ge1}\left|\frac{n^2z^2}{k^2\(n^2z^2-k^2\)}\right| \le\sum_{k\ge1}\frac{2n^2y^2}{k^4} \le 3n^2y^2.
$$
Therefore, it follows that
\begin{equation}\label{h(q)}
h_{A,B} (q)=  \frac{-1}{2\pi iz} f_{1,2A,2B-1}(z)-\frac{z\pi}{12i}f_{-1,2A,2B-1}(z)-\frac{z}{\pi i} \sum_{n\ge1}(-1)^n n q^{An^2+\left(B-\frac12\right)n}S_{n}(z),
\end{equation}
where $\left|S_{n}(z)\right|\le 3n^2 y^2$ and  $f_{j,a,b}(z):=\sum_{n=1}^{\infty}(-1)^n n^{-j} q^{\frac{an^2+bn}{2}}$. For nonnegative integers $j$, the asymptotic behavior of $f_{j,a,b} (z)$ is well-known (See \cite{BK, Kor}). By adopting the same technique in \cite{BK} using Zagier's asymptotic expansion \cite{DZ}, we find that $f_{-1,a,b} (z) = \frac{1}{4} + O(y)$ for $|x|\le y$ and $y\rightarrow 0^+$. Note also that 
\begin{equation}\label{Snk}
\left|\sum_{n\ge1}(-1)^n n q^{An^2+\left(B-\frac12\right)n}S_{n}(z) \right|\le 3y^2\sum_{n\ge1} n^3 e^{-2\pi y\(An^2+\(B-\frac12\)n\)} \ll 1.
\end{equation}
Since $f_{1,a,b}(z) = - \log 2 + O(y)$ for $|x|\le y$ and $y\rightarrow 0^+$, we have proven the following.

\begin{lemma}\label{hAB1lemma}
For $|x|\le y$, as $y\rightarrow 0^+$
\[
h_{A,B} (q ) = \frac{\log2}{2\pi iz}  + O( 1).
\]
\end{lemma}

Now we turn to investigate $h_{A,B} (z)$ near $q=-1$.  By setting $\tau:=z-\frac12=x-\frac12+iy$ and $Q:=e^{2\pi i\tau}=-q$, we derive that 
\begin{align*}
h_{A,B}(q)=h_{A,B}(-Q)&=\sum_{n\ge1}(-1)^n\frac{(-Q)^{An^2+Bn}}{1-(-Q)^{n}}=\sum_{n\ge1}\frac{(-1)^{(An+B+1)n}Q^{An^2+Bn}}{1-(-1)^nQ^n}\\
&= \sum_{n \ge 1} \frac{ (-1)^n Q^{4 An^2 + 2Bn}}{1-Q^{2n}} + (-1)^{A+B} \sum_{n \geq 1} \frac{ (-1)^n Q^{4An^2 -(4A -2B)n + A-B}}{1+Q^{2n-1}},
\end{align*}
provided $B$ is a half-integer. The first sum can be estimated using Lemma \ref{hAB1lemma}, and thus {we need only deal with the second sum.

By setting $w=(2n-1)\tau + 1/2$ in \eqref{Mittag}, we find that 
\begin{align*}
\frac{e^{\pi i (2n-1)\tau } i }{1+e^{2\pi i (2n-1)\tau}} &= \frac{1}{-2\pi i} \frac{1}{(2n-1)\tau+1/2} \\
&\qquad\qquad+ \sum_{k=1}^{\infty} \frac{(-1)^k}{-2\pi i} \( \frac{1}{(2n-1)\tau-k +1/2} + \frac{1}{(2n-1)\tau+k +1/2} \) \\
&= \sum_{k=1}^{\infty} \frac{(-1)^k}{-2\pi i} \( \frac{1}{(2n-1)\tau-k +1/2} - \frac{1}{(2n-1)\tau+k - 1/2} \) \\
&= \sum_{k=1}^{\infty} \frac{(-1)^k}{-\pi i} \( \frac{1}{(4n-2)\tau- (2k -1)} - \frac{1}{(4n-2)\tau+2k - 1} \) \\
&= \sum_{k=1}^{\infty} \frac{(-1)^k}{-\pi i}  \frac{4k-2}{(4n-2)^2\tau^2- (2k -1)^2}.
\end{align*}
Applying this, we have
\begin{multline*}
 \sum_{n \geq 1} \frac{ (-1)^n Q^{4An^2 -(4A -2B)n + A-B}}{1+Q^{2n-1}}= \sum_{n \geq 1} (-1)^n Q^{4An^2 -(4A -2B+1)n + A-B+\frac12} \cdot \frac{ Q^{n-\frac12} }{1+Q^{2n-1}}\\
=\frac{1}{\pi}\sum_{n \geq 1} (-1)^n Q^{4An^2 -(4A -2B+1)n + A-B+\frac12} \sum_{k=1}^{\infty} (-1)^k  \frac{4k-2}{(4n-2)^2\tau^2- (2k -1)^2}.
\end{multline*}
By decomposing
\[
\frac{1}{(4n-2)^2\tau^2-(2k-1)^2} = \(\frac{1}{(4n-2)^2\tau^2-(2k-1)^2}+\frac{1}{(2k-1)^2}-\frac{1}{(2k-1)^2}\)
\]
and noting that
\begin{multline*}
\left| \frac{1}{(4n-2)^2\tau^2-(2k-1)^2}+\frac{1}{(2k-1)^2} \right| \\
= \left| \frac{(4n-2)^2\tau^2}{(2k-1)^2\((4n-2)^2\tau^2-(2k-1)^2\)}\right| 
 \le  \frac{2(4n-2)^2y^2}{(2k-1)^4},
\end{multline*}
we conclude that the whole sum is bounded.

\begin{lemma}\label{hAB-1lemma} For $\left|x-\frac12\right|\le y$ and a half-integer $B$, as $y\rightarrow 0^+$, 
\[
h_{A,B} (q) = \frac{\log 2}{4 \pi i \tau} + O\( 1 \).
\]
\end{lemma}

Since the term $\frac{1}{(q^2;q^2)_\infty}$ in \eqref{genSC2} is exponentially small away from the dominant poles, we do not need a sharp bound for $h_{A,B} (z)$ in this region.

\begin{lemma}\label{haway}
For $y>0$ with $y\le |x| \le 1/2-y$, 
\[
\left| h_{A,B}(q)\right| \ll y^{-3/2}.
\]
\end{lemma}

\begin{proof}
$$
\left| h_{A,B}(q)\right| \le\frac{1}{1-|q|}\sum_{n\ge1}|q|^{An^2+Bn} \ll \frac{1}{y} y^{-1/2}.
$$
\end{proof}

\section{Proof of Theorem \ref{main1thm}}
In this section, we use Wright's Circle Method to complete the proof of Theorem \ref{main1thm}. Before beginning the proof, we first investigate $\frac{1}{(q^2;q^2)_{\infty}}$ near and away from $q=\pm 1$. Recall that, from the modular inversion formula for Dedekind's eta-function (\cite[P.121, Proposition 14]{Kob}), 
\begin{equation}\label{qasym}
(q;q)_\infty=\frac{1}{\sqrt{-iz}}e^{-\frac{\pi iz}{12}-\frac{\pi i}{12z}}\(1+O\(e^{-\frac{2\pi i}{z}}\)\).
\end{equation}
Therefore, we find that
\begin{equation}\label{prodesti}
\begin{aligned}
\frac{1}{(q^2;q^2)_\infty}&=\sqrt{-2iz}\ e^{\frac{\pi i}{24z}} +O\(y^{3/2}e^{\frac{\pi}{24} \text{Im} \( \frac{-1}{z} \) }  \), \quad \text{ for $|x| < y,$ } \\
\frac{1}{(Q^2;Q^2)_\infty}&=\sqrt{-2i\tau}\ e^{\frac{\pi i}{24\tau}} +O\(y^{3/2}e^{\frac{\pi}{24} \text{Im} \( \frac{-1}{\tau} \) }  \),\quad \text{ for $|x-1/2|<y$.}
\end{aligned}
\end{equation}

Now we consider the behavior away from the dominant poles, i.e., in a range of $y\le |x| \le 1/2-y$. Note that
$$
\log\(\frac{1}{(q^2;q^2)_\infty}\)=-\sum_{n=1}^\infty\log\(1-q^{2n}\)=\sum_{n=1}^\infty\sum_{m=1}^\infty\frac{q^{2nm}}{m}=\sum_{m=1}^\infty\frac{q^{2m}}{m\(1-q^{2m}\)}.
$$
Thus,
\begin{align}
\left|\log\(\frac{1}{(q^2;q^2)_\infty}\)\right|\le \sum_{m=1}^\infty\frac{|q|^{2m}}{m\left|1-q^{2m}\right|}&\le \sum_{m=1}^\infty\frac{|q|^{2m}}{m\(1-|q|^{2m}\)}+\frac{|q|^2}{\left|1-q^2\right|}-\frac{|q|^2}{1-|q|^2}\nonumber\\
&=\log\(\frac{1}{(|q|^2;|q|^2)_\infty}\)-|q|^2\(\frac{1}{1-|q|^2}-\frac{1}{\left|1-q^2\right|}\).\label{second}
\end{align}
Plugging $z\mapsto 2iy$ into \eqref{qasym}, we find that
\[
\log\(\frac{1}{(|q|^2;|q|^2)_\infty}\)= \frac{\pi}{24y}+\frac12\log(2y)+O(y).
\]
To estimate the other term in \eqref{second}, first note that if $y\le |x| \le \frac14$, then $\cos(4\pi y)\ge\cos(4\pi x)$. On the other hand, if $\frac14\le |x| \le \frac12-y$, then $\cos(4\pi x)\le\cos(2\pi -4\pi y)=\cos(4\pi y)$. Thus, for all $y\le |x| \le 1/2-y$, $\cos(4\pi x)\le\cos(4\pi y)$. Therefore,
\[
\left|1-q^2\right|^2=1-2e^{-4\pi y}\cos(4\pi x)+e^{-8\pi y}\ge 1-2e^{-4\pi y}\cos(4\pi y)+e^{-8\pi y}.
\]
From the Taylor expansion, we conclude that $\left|1-q^2\right|\le4\sqrt{2}\pi y+O\(y^2\)$. Since $1-|q|^2=1-e^{-4\pi y}=4\pi y+O\(y^2\)$, we arrive at
\begin{equation}\label{prodaway}
\left|\frac{1}{(q^2;q^2)_\infty}\right|\ll \sqrt{2y}\exp\left[ \frac{1}{y}\(\frac{\pi}{24}-\frac{1}{4\pi}\(1-\frac{1}{\sqrt{2}}\)\) \right].
\end{equation}

Now we are ready to prove Theorem \ref{main1thm}. 

\begin{proof}[Proof of Theorem \ref{main1thm}]
First, rewrite \eqref{genSC1} in terms of $h_{A,B}(q)$ as follows:
\begin{equation}\label{SC1-h}
S_{C_1, m} (q) = \sum_{n \geq 0} N_{C_1} (m,n) q^n = \frac{-1}{(q^2 ; q^2)_{\infty}} \(h_{\frac12,\frac{1+2|m|}{2}}(q)- h_{\frac32,\frac{1+|m|}{2}}\(q^2\)\).
\end{equation}
By Cauchy's Theorem, we see for $y=\frac{1}{4\sqrt{3n}}$ that
\begin{align*}
N_{C_1} (m,n)&=\frac{1}{2\pi i}\int_\mathcal{C}\frac{S_{C_1, m} (q)}{q^{n+1}}dq=\int_{-\frac12}^{\frac12} S_{C_1,m}\(e^{2\pi ix-\frac{\pi}{2\sqrt{3n}}}\)e^{-2\pi inx+\frac{\pi\sqrt{n}}{2\sqrt{3}}} dx\\
&=\int_{|x|\le y}S_{C_1, m} \(e^{2\pi ix-\frac{\pi}{2\sqrt{3n}}}\)e^{-2\pi inx+\frac{\pi\sqrt{n}}{2\sqrt{3}}}  dx \\
&\qquad+\int_{y\le|x|\le\frac12-y}S_{C_1,m}\(e^{2\pi ix-\frac{\pi}{2\sqrt{3n}}}\)e^{-2\pi inx+\frac{\pi\sqrt{n}}{2\sqrt{3}}}dx\\
&\qquad\qquad+\int_{\left|x-\frac12\right|\le y}S_{C_1,m}\(e^{2\pi ix-\frac{\pi}{2\sqrt{3n}}}\)e^{-2\pi inx+\frac{\pi\sqrt{n}}{2\sqrt{3}}}dx\\
&=:\mathcal{I}_1+\mathcal{I}_2+\mathcal{I}_3,
\end{align*}
where $\mathcal{C}=\{|q|=e^{-\frac{\pi}{2\sqrt{3n}}}\}$. In this case, the integral $\mathcal{I}_1$ contributes the main term and the integrals $\mathcal{I}_2$ and $\mathcal{I}_3$ are absorbed in the error term.

To evaluate $\mathcal{I}_1$ we first introduce a function $P_s(u)$ which is defined by Wright \cite{Wright}. For fixed $M>0$ and $u\in\mathbb{R}^+$
$$
P_{s}(u):=\frac{1}{2\pi i} \int_{1-Mi}^{1+Mi} v^s e^{u\(v+\frac{1}{v}\)} dv
$$
This functions is rewritten in terms of the $I$-Bessel function up to an error term. 
\begin{lemma}[\cite{Wright}] As $n\rightarrow \infty$
$$
P_s(u)=I_{-s-1}(2u)+O\(e^u\),
$$
where $I_\ell$ denotes the usual the $I$-Bessel function of order $\ell$.
\end{lemma}

From Lemma \ref{hAB1lemma} and \eqref{prodesti}, we find that  for $|x|\le \frac{1}{4\sqrt{3n}}=y$ as $n\rightarrow \infty$
$$
S_{C_1, m} (q) =  \frac{e^{\frac{\pi i}{24z}}\log2}{2\pi\sqrt{-2iz}}+O\(n^{-1/4}e^{\frac{\pi}{24}\text{Im}\(-\frac{1}{z}\)}\).
$$
Thus the integral $\mathcal{I}_1$ becomes
$$
\int_{|x|\le \frac{1}{4\sqrt{3n}}}  \(\frac{e^{\frac{\pi i}{24z}}\log2}{2\pi\sqrt{-2iz}}+O\(n^{-1/4}e^{\frac{\pi}{24}\text{Im}\(-\frac{1}{z}\)}\) \)  e^{-2\pi inx+\frac{\pi\sqrt{n}}{2\sqrt{3}}}  dx.
$$
By making the change of variables $v= 1- i 4\sqrt{3n}x$, we arrive at
\begin{align*}
\int_{1-i}^{1+i} &\frac{1}{i4\sqrt{3n}}\(\frac{e^{\frac{\pi\sqrt{n}}{2\sqrt{3}v}}(3n)^{1/4} \log2}{\pi\sqrt{2v}}+O\(n^{-1/4}e^{\frac{\pi\sqrt{n}v}{2\sqrt{3}}} \)\)e^{\frac{\pi \sqrt{n} v}{2\sqrt{3}}} dv\\
&=\frac{(3n)^{-1/4}\log2}{2\sqrt{2}}P_{-\frac12}\(\frac{\pi\sqrt{n}}{2\sqrt{3}}\)+O\(n^{-3/4}e^{\frac{\pi\sqrt{n}}{\sqrt{3}}}\)\\
&=\frac{(3n)^{-1/4}\log2}{2\sqrt{2}}I_{-\frac32}\(\frac{\pi\sqrt{n}}{\sqrt{3}}\)+O\(n^{-3/4}e^{\frac{\pi\sqrt{n}}{\sqrt{3}}}\)\\
&=\frac{\log2}{4\pi}n^{-1/2}e^{\frac{\pi\sqrt{n}}{\sqrt{3}}}+O\(n^{-3/4}e^{\frac{\pi\sqrt{n}}{\sqrt{3}}}\),
\end{align*}
where we use the asymptotic formula for the $I$-Bessel function \cite[4.12.7]{AAR}
\[
I_\ell(x)=\frac{e^x}{\sqrt{2\pi x}}+O\(\frac{e^x}{x^{\frac32}}\).
\]

Now we consider the integral $\mathcal{I}_2$. From the Lemma \ref{haway} and \eqref{prodaway}, for $\frac{1}{4\sqrt{3n}}\le|x|\le\frac12-\frac{1}{4\sqrt{3n}}$ we have
$$
S_{C_1, m} (q)\ll n^{1/2}\text{exp}\(\frac{\pi\sqrt{n}}{2\sqrt{3}}-\frac{\sqrt{3n}}{\pi}\(1-\frac{1}{\sqrt{2}}\) \),
$$
 as $n\rightarrow \infty$. Hence, we have
$$
\mathcal{I}_2\ll n^{1/2}\text{exp}\(\frac{\pi\sqrt{n}}{\sqrt{3}}-\frac{\sqrt{3n}}{\pi}\(1-\frac{1}{\sqrt{2}}\) \).
$$

Finally, to estimate $\mathcal{I}_3$ first we shift $x\mapsto \widetilde{x}+\frac12$ (thus $\tau=\widetilde{x}+i\frac{\pi}{4\sqrt{3n}}$), and rewrite $\mathcal{I}_3$ as follows:
$$
\mathcal{I}_3=\int_{\left|x-\frac12\right|\le y}S_{C_1, m} \(e^{2\pi iz}\)e^{-2\pi inx+\frac{\pi\sqrt{n}}{2\sqrt{3}}} dx
=(-1)^n\int_{|\widetilde{x}|\le y}S_{C_1, m} \(e^{2\pi i\tau}\)e^{-2\pi in\widetilde{x}+\frac{\pi\sqrt{n}}{2\sqrt{3}}} d\widetilde{x}
$$
As before, from Lemma \ref{hAB-1lemma} and \eqref{prodesti}, we find that  for $\left|x-\frac12\right|=|\widetilde{x}|\le \frac{1}{4\sqrt{3n}}$  as $n\rightarrow \infty$
$$
S_{C_1, m} \(e^{2\pi i\tau}\)= O\(y^{1/2} e^{\frac{\pi}{24} \text{Im} \( \frac{-1}{\tau} \)} \).
$$
Thus, we arrive at
$$
\mathcal{I}_3\ll \int_{|\widetilde{x}|\le\frac{1}{4\sqrt{3n}}} n^{-1/4} e^{\frac{\pi\sqrt{n}}{\sqrt{3}}}  d\widetilde{x} \ll n^{-3/4}e^{\frac{\pi\sqrt{n}}{\sqrt{3}}},
$$
which finishes the proof of Theorem \ref{main1thm}.
\end{proof}

By noting that 
\[
S_{C_5} (q) = \sum_{n \geq 0} N_{C_5} (m,n) q^n = \frac{-1}{(q^2;q^2)_{\infty}} \( h_{\frac12,\frac{1+2|m|}{2}}(q) - h_{\frac12,\frac{1+2|m|}{2}}(q^2)\),
\]
we can deduce the following asymptotic formula by proceeding in the same way as before. 

\begin{proposition}
As $n \to \infty$,
\[
N_{C_5} (m,n)  \sim \frac{\log2}{4\pi \sqrt{n} } e^{\frac{\pi\sqrt{n}}{\sqrt{3}}}.
\]
\end{proposition}

\section{Proof of Theorem \ref{main2thm}}

In this section, we study crank differences and prove their sign pattern.  From \eqref{genSC1}, one easily sees that for a nonnegative integer $m$
\begin{equation}\label{SC1diff}
\begin{aligned}
SD_{C_1,m} (q) &:= S_{C_1, m} (q) - S_{C_1, m+1} (q) \\
&= \frac{1}{(q^2 ; q^2)_{\infty}} \sum_{n \geq 1} (-1)^{n-1} \( q^{n(n+1)/2 + mn} - q^{n(3n+1) + 2mn} \).
\end{aligned}
\end{equation}

In \cite[Theorem 1.1]{kks}, the asymptotic behavior of $f_{0,a,b}$ is effectively given. It implies    that for $|x| < y$ as $y \to 0^{+}$,
\[
f_{0,a,b}(z)=-\frac12+\frac{b}{8}(-2\pi iz)+O(y^2).
\]
From the above, we easily see that for $|x| < y$ as $y \to 0^{+}$, 
\begin{equation}\label{SD1near1}
SD_{C_1,m} (q) = -\frac{(1+2m)\pi\sqrt{2i}}{4} z^{3/2}e^{\frac{\pi i}{24z}} + O \( y^{5/2}e^{\frac{\pi}{24} \text{Im}\(\frac{-1}{z}\) } \).
\end{equation}

For asymptotic behavior near $q=-1$, we set $\tau=z-\frac12=x-\frac12+iy$ and $Q=e^{2\pi i\tau}=-q$ as before. Then, we obtain that 
\begin{align}
SD_{C_1,m} (q) &=\frac{1}{(Q^2 ; Q^2)_{\infty}}\(-f_{0,4,2(2m+1)}(\tau)+(-1)^mQ^{-m}f_{0,4,2(2m-1)}(\tau)+f_{0,6,2(2m+1)}(\tau)\)\nonumber\\
&=\(\sqrt{-2i\tau}e^{\frac{\pi i}{24\tau}}+O\(y^{3/2}e^{\frac{\pi}{24}\text{Im}\(\frac{-1}{\tau}\)}\)\)\(-\frac{(-1)^m}{2}+\frac{(-1)^m}{2}\pi i\tau+O\(y^2\)\)\nonumber\\
&=\frac{(-1)^{m+1}}{2} \sqrt{-2i\tau} e^{\frac{\pi i}{24\tau}} + O \( y^{3/2}e^{\frac{\pi}{24} \text{Im}\(\frac{-1}{\tau}\) } \), \label{SD1near-1}
\end{align}
as $y \to 0^{+}$ for $|x-1/2| \leq y$. The rest of proof is almost identical to that of Theorem \ref{main1thm} except that $\mathcal{I}_3$ contributes the main term this time. Since the other estimates are similar, we only give a brief explanation for the main term. In this case, $\mathcal{I}_3$ becomes 
$$
\mathcal{I}_3=\int_{\left|x-\frac12\right|\le y}SD_{C_1, m} \(e^{2\pi iz}\)e^{-2\pi inx+\frac{\pi\sqrt{3n}}{6}} dx
=(-1)^n\int_{|\widetilde{x}|\le y}SD_{C_1, m} \(e^{2\pi i\tau}\)e^{-2\pi in\widetilde{x}+\frac{\pi\sqrt{3n}}{6}} d\widetilde{x}.
$$
As before, by setting $v= 1- i 4\sqrt{3n}\widetilde{x}$ and the equation \eqref{SD1near-1}, we find that
\[
\mathcal{I}_{3} \sim  \frac{(-1)^{n+m+1} \pi (3n)^{-3/4}}{4\sqrt{2}} P_{\frac12} \( \frac{\pi \sqrt{n}}{2\sqrt{3}} \) 
 \sim \frac{(-1)^{n+m+1} }{8\sqrt{3} n} e^{\frac{\pi \sqrt{n}}{\sqrt{3}}} ,
\]
which completes the proof of Theorem \ref{main2thm}.

Again, we can also think about the differences for $N_{C_5}(m,n)$. From \eqref{genSC2} it follows that
\[
\begin{aligned}
SD_{C_5,m} :&= S_{C_5, m} (q) - S_{C_5, m+1} (q) \\
&= \frac{1}{(q^2 ; q^2)_{\infty}} \sum_{n \geq 1} (-1)^{n-1} \( q^{n(n+1)/2 + |m|n} - q^{n(n+1) + 2|m|n} \).
\end{aligned}
\]
By proceeding the same way, we can also conclude that for a fixed nonnegative integer $m$,
\[
(-1)^{n+m+1} (N_{C_5} (m,n) - N_{C_5} (m+1,n) ) > 0
\]
holds for large enough integers $n$.

\section{Modularity and congruences}

In this section, we show that the generating function for $\operatorname{spt}_{\omega} (n)$ is a part of the holomorphic part of a certain weight $3/2$ harmonic weak Maass form to prove the Theorem \ref{main3thm}.

We first note that  from \cite[Lemma 6.1]{ADY} 
\begin{align*}
S_\omega(z):=\sum_{n\ge1}\s(n)q^n&=\sum_{n\ge1}\frac{q^n}{\left(1-q^n\right)^2\left(q^{n+1};q\right)_n\left(q^{2n+2};q^2\right)_\infty}\\
&=\frac{1}{(q^2;q^2)_\infty}\sum_{n\ge1}\frac{nq^n}{1-q^n}+\frac{1}{\left(q^2;q^2\right)_\infty}\sum_{n\ge1}\frac{(-1)^n\left(1+q^{2n}\right)q^{n(3n+1)}}{\left(1-q^{2n}\right)^2}.
\end{align*}
We also note that 
$$
E_2(z):=1-24\sum_{n=1}^\infty\sigma_1(n)q^n=1-24\sum_{n\ge1}\frac{nq^n}{1-q^n},
$$
as usual the Eisenstein series with $\sigma_1(n):=\sum_{d\mid n}d$, and
$$
R_2(z):=\sum_{n\ge0}\frac12N_2(n)q^n=\frac{-1}{(q;q)_\infty}\sum_{n\ge1}\frac{(-1)^n(1+q^n)q^{\frac{n(3n+1)}{2}}}{(1-q^n)^2},\quad\text{\cite[(3.4)]{GA}},
$$
with $N_j(n):=\sum_{m\in\Z}m^jN(m,n)$ the moments of the rank. By using these, we rewrite $S_\omega(z)$ as follows:
$$
S_\omega(z)=\frac{q^{\frac{1}{12}}\left(1-E_2(z)\right)}{24\eta(2z)}-R_2(2z),
$$
where $\eta(z):=q^{\frac{1}{24}}\prod_{n=1}^\infty(1-q^n)$ is Dedekind's $\eta$-function.
To state more define (we have corrected the sign for $\M(z)$ in \cite{KB} (c.f. \cite{fo}).)
\begin{align*}
\N(z)&:=\frac{i}{4\sqrt{2}\pi}\int_{-\overline{z}}^{i\infty}\frac{\eta(24z)}{(-i(\tau+z))^{\frac32}}d\tau,\\
\M(z)&:=q^{-1}S_\omega(12z)-\frac{E_2(24z)}{24\eta(24z)}-\N(z).
\end{align*}
\begin{lemma}
The function $\M(z)$ is a harmonic weak Maass form of weight $3/2$ on $\Gamma_0(576)$ with Nebentypus character $\chi_{12}(\cdot):=\left(\frac{12}{\cdot}\right)$. 
\end{lemma}}
It is an immediate result from  \cite[Theorem 1.1]{KB} and the basic properties of $E_2 (z)$ and $\eta (z)$.
Here  $\N(z)$ is the non-holomorphic part and supported on finitely many square classes. Hence, the holomorphic part
$
q^{-1}S_\omega(12z)-\frac{E_2(24z)}{24\eta(24z)}
$
is a mock modular form of weight $3/2$ and becomes an weakly holomorphic modular form with appropriate arithmetic progressions which make  $\N(z)$ vanish. Theorem \ref{main3thm} follows from \cite[Theorem 1.1]{CON} together with the fact that $E_2(z)$ is a $p$-adic modular form for any prime $p$.

\end{document}